\documentclass[12pt]{article}

\usepackage{amssymb}
\usepackage{amsmath}
\usepackage{comment}
\usepackage{enumerate}
\usepackage{amsthm}
\usepackage[flushmargin]{footmisc}

\newtheorem{theorem}{Theorem}[section]
\newtheorem{lemma}[theorem]{Lemma}
\newtheorem{prop}[theorem]{Proposition}
\theoremstyle{definition}

\theoremstyle{remark}

\numberwithin{equation}{section}

\newcommand{\Q}{\mathbf Q}

\newcommand{\Z}{\mathbf Z}

\newcommand{\Gal}{\mathrm{Gal}}

\newcommand{\Hol}{\mathrm{Hol}}

\newcommand{\Sym}{\operatorname{Sym}}

\newcommand{\GL}{\mathrm{GL}}

\newcommand{\SL}{\mathrm{SL}}

\newcommand{\End}{\operatorname{End}}

\newcommand{\Aut}{\operatorname{Aut}}

\newcommand{\Id}{\operatorname{Id}}

\newcommand{\al} {{\alpha}}
\newcommand{\la} {{\lambda}}
\newcommand{\wk} {{\widetilde{K}}}
\newcommand{\ii} {{\imath}}

\begin{document}

\title{From Galois to Hopf Galois: \\ theory and practice}

\author{Teresa Crespo, Anna Rio, Montserrat Vela}

\maketitle

\let\thefootnote\relax\footnote{T. Crespo acknowledges support by grants MTM2012-33830, Spanish
Science Ministry, and 2009SGR 1370; A.Rio and M. Vela acknowledge
support by grants MTM2012-34611, Spanish Science Ministry, and
2009SGR 1220. }

\begin{abstract}
Hopf Galois theory expands the classical Galois theory by considering the Galois property in terms of
the action of the group algebra $k[G]$ on $K/k$ and then replacing it by the action of a Hopf algebra.
We review the case of separable extensions where the Hopf Galois property admits a group-theoretical
formulation suitable for counting and classifying, and also to perform explicit computations and
explicit descriptions of all the ingredients involved in a Hopf Galois structure. At the end we give
just a glimpse of how this theory is used in the context of Galois module theory for wildly ramified
extensions.
\end{abstract}

\section{Introduction}

A Galois extension is an algebraic field extension $K/k$ that is normal and separable.
The significance of being a Galois extension is that
$K/k$  has a Galois group $G$ and obeys the fundamental theorem of Galois theory:
there is a one-to-one correspondence between the lattice of its intermediate fields and the lattice of subgroups of $G$.

Hopf Galois theory arises as an attempt to expand classical Galois theory to more general settings.
In order to enlarge the category of algebraic objects attached to field extensions, since we have a fixed base
field $k$,  the group $G$ leads to the group algebra $k[G]$, which is a cocommutative Hopf algebra with
comultiplication $\Delta(g) = g\otimes g$, counit $\epsilon(g)=1$  and antipode $S(g)=g^{-1}$, for all $g \in G$.  The
essential requirement to proceed is then fulfilled: the lattice of sub-Hopf algebras of $k[G]$ is in
one-to-one correspondence with the lattice of subgroups of $G$.

Putting this machinery to work, the Galois action of $G$ in $K$ as automorphism group extends linearly to an action
$k[G]\times K\to K$ which provides a Hopf action $\mu:k[G]\to \End_k(K)$. Then, the condition of being
a Galois extension can be reformulated in the following way:
$$K/k \mbox{ is  Galois } \iff (1,\mu):  K \otimes_k k[G] \rightarrow \End_k(K)  \mbox{ is an isomorphism, }$$
where $(1,\mu)(s\otimes h)(t)= s \cdot(\mu(h)(t))$. Now, in order to generalize we just have to replace $k[G]$ by an
object of a suitable algebraic category. From now on, we restrict ourselves to the case of finite extensions.

The concept of Hopf Galois extension is due to Chase and Sweedler \cite{ChaseSw}:
if $K/k$ is a finite extension of fields, we say that $K/k$ is a \emph{Hopf Galois extension} if there exists
a finite cocommutative $k-$Hopf algebra ${H}$
and a Hopf  action $\mu: H\to \End_k(K)$ 
such that
$$(1,\mu):  K \otimes_k H \rightarrow \End_k(K) \mbox{ is an } \mbox{isomorphism.}$$ 
That is, $K$ is an $H$-module and the endomorphisms of $K$ are all obtained from the homotheties and the Hopf action.
From this definition we get $\dim H=[K:k]$.

In the Hopf Galois setting the following \emph{fundamental theorem} holds:

\begin{theorem}[\cite{ChaseSw} Theorem 7.6] Let $K/k$ be a Hopf Galois extension with algebra $H$ and  Hopf action $\mu: H\to \End_k(K)$.
For a $k$-sub-Hopf algebra $H'$ of $H$ we define
$$
K^{H'}=\{x\in K \mid \mu(h)(x)=\epsilon(h)\cdot x \mbox{ for all } h\in H'\},
$$
where $\epsilon$ is the counity of $H$.
Then, $K^{H'}$ is a subfield of $K$, containing $k$, and
$$
\begin{array}{rcl}
{\mathcal F}_H:\{H'\subseteq H \mbox{ sub-Hopf algebra}\}&\longrightarrow&\{\mbox{Fields }E\mid k\subseteq E\subseteq K\}\\
H'&\to &K^{H'}
\end{array}
$$
is injective and inclusion reversing.
\end{theorem}

The theory of Hopf Galois extensions was first considered to study purely inseparable field extensions (see \cite{ChaseSw}). Chase found that the fundamental theorem
of Galois theory in its strong form does not hold for Hopf Galois structures on purely inseparable extensions of exponent $>1$. That led Chase in \cite{Chase} (and later in \cite{Chase2}) to develop a fundamental theorem
of Galois theory for purely inseparable field extensions where the Hopf Galois action is by the Hopf algebra $H_t$ representing the truncated automorphism scheme of $K/k$. But $K/k$ is not a $H_t$-Hopf Galois extension, beacause if $[K:k]=n$, then $\dim_k(H_t)=n^n$.

Greither and Pareigis \cite{GP} recovered the notion of Hopf Galois extension to look at separable extensions.
When we deal with separable field extensions, the technique of Galois descent shows that the property of being Hopf Galois is encoded in
the Galois group of the normal closure.  If we assume that $K/k$ is separable and
Hopf Galois, then $(\tilde K\otimes_k K)/\tilde K$ is also Hopf Galois, where $\tilde K$ is the normal closure of $K/k$. To prove this
one considers the $\tilde K-$Hopf algebra $\tilde K\otimes _k H$. If we denote $G=\Gal(\tilde K/k)$, then the action of $H$
on $K$ is recovered by identifying $H$ and $K$ with the fixed rings $(\tilde K\otimes _k H)^G$ and $(\tilde K\otimes_k K)^G$,
where $G$  acts on the left factor as automorphism group. This leads to the Greither and Pareigis
characterization and classification of Hopf Galois structures on separable field extensions, achieved by transforming
the problem into a group-theoretic problem involving the Galois group $G$.

\begin{theorem}[\cite{GP} Theorem 2.1]\label{defHG}
Let $K/k$ be a separable extension of degree $n$ and let $\tilde K/k$ be its Galois closure, $G=\Gal(\tilde K/k)$
and $G'=\Gal(\tilde K/K)$.

$K/k$ is a Hopf Galois extension if, and only if, there exists a regular subgroup $N$ of $S_n$ normalized by $G$, where $G$ is identified as a subgroup of $S_n$ via the action of $G$ on the left cosets $G/G'$.
\end{theorem}

The identification of $G$ as a subgroup of $S_n$ mentioned in the theorem is given by
$$
 \begin{array}{rl}
\lambda: & G\rightarrow \Sym(G/G')\\
 &g\rightarrow (\lambda_g: xG'\mapsto gxG')\, .
\end{array}
$$
Any enumeration of the left cosets provides an identification of $G$ and $\lambda(G)$ as a transitive subgroup of the symmetric group $S_n$.

For a separable extension $K/k$, say that \emph{the fundamental theorem of Galois theory holds in its strong form for $K/k$} if there exists a Hopf Galois structure on $K/k$ for which the strong form holds. Clearly the fundamental theorem of Galois theory holds in its strong form for every Galois extension $K/k$. In \cite{GP} a class of non-Galois extensions is identified for which the strong form also holds. We say
that $K/k$ is an \emph{almost classically Galois extension} if there exists a regular subgroup $N$ of $S_n$ normalized by $G$ and contained in $G$, where $G$ is identified as a subgroup of $S_n$ as in Theorem \ref{defHG}

\begin{theorem}[\cite{GP} 4.1]
Let $K/k$ be a separable extension of degree $n$ and let $\tilde K/k$ be its Galois closure, $G=\Gal(\tilde K/k)$
and $G'=\Gal(\tilde K/K)$.

$K/k$ is almost classically Galois if, and only if, $G'$ has a normal complement $N$ in $G$.
\end{theorem}
In particular, if $K/k$ is Galois, then $G'=1$ and has normal complement $N=G$.
The following theorem provides a justification for the notion of almost classically Galois extensions.

\begin{theorem}[\cite{GP} 5.2]
If $K/k$ is almost classically Galois, then there is a Hopf algebra $H$ such that $K/k$ is Hopf Galois with algebra $H$ and the main theorem holds
in its strong form, namely there is a bijective correspondence between $k$-sub-Hopf algebras of $H$ and $k$-subfields of $K.$
\end{theorem}

In some sense, almost classically Galois extensions are too similar to classical Galois extension and to get a better understanding of
the significance of the Hopf Galois property we should work with separable extensions being Hopf Galois but not almost classically Galois.
In degree $\le 7$, there are no such extensions, as we shall show in more detail in sections 2 and 4 below. The smallest example can be found in  degree 8 over the rational field, as we show in \cite{CRV1}. An example of degree 16 was constructed in \cite{GP}, where the base field $k$ is a quadratic extension of $\Q$. In \cite{CRV2}, we prove that the class of extensions for which the fundamental theorem of Galois theory holds in its strong form is larger than the class of almost classically Galois extensions by constructing a non-almost classically Galois extension for which the strong form holds.

Helpful for deciding the existence of a Hopf Galois structure on $K/k$ is a reformulation of Theorem \ref{defHG}, due to Childs, that reverses the relationship between $G$ and $N$: instead of looking
 for regular subgroups of $S_n$ normalized by $G$ one should look for embeddings of $G$ into the holomorph $\Hol(N)=N\rtimes \Aut N$
 of a group $N$ of order $n$ (see \cite{Childs1} Proposition 1). The group $\Hol(N)$ has a natural embedding in $\Sym(N)\simeq S_n$. Since $\Hol(N)$ is much smaller than $S_n$, this breaks the problem into a collection of problems, parametrized by
the isomorphism classes of groups of order $n$ and more suitable to be considered for a systematic computational treatment.

\begin{theorem}\label{hol}
Let $K/k$ be a separable extension of degree $n$ and let $\tilde K/k$ be its Galois closure, $G=\Gal(\tilde K/k)$
and $G'=\Gal(\tilde K/K)$.
$K/k$ is a Hopf Galois extension if, and only if, there exists a group $N$ of order $n$  such that
$G\subseteq \Hol(N)$, where $G$ is identified as a subgroup of $S_n$ via the action of $G$ on the left cosets $G/G'$.
\end{theorem}

As an example of the significance of this reformulation, let us assume that $K/k$ is a separable extension  of degree 6 with Galois group isomorphic to $S_4$. We consider $\Hol(C_6)$ and $\Hol(S_3)$.   Since the first one has order 12 and the second one has order 36, we immediately conclude that
$K/k$ is not Hopf Galois.

Theorem \ref{hol} might be considered as an ``algorithmic" procedure
to check if the separable extension $K/k$ is a Hopf Galois extension:
\begin{description}
\item[Step 0] Check if $G'$ has a normal complement $N$ in $G$. In that case, $K/k$ is almost classically Galois.

\item[Step 1] Let $N$ run through a system of representatives of isomorphism classes of groups of order $n$
\item[Step 2] Compute $\Hol(N)\subseteq S_n$
\item[Step 3] Check $G\subseteq \Hol(N)\subset \Sym(G/G')$ such that $G'\subset G$ is the stabilizer of a point.
\end{description}

In degrees $n\le 5$, Greither and Pareigis showed that all Hopf Galois extensions are either Galois or almost classically Galois and they completely describe the Hopf Galois character of $K/k$ according to the Galois group (or the degree) of $\wk/k$.
The above algorithmic procedure allowed us to go further and proceed with the small case $n=6$ to see
how the Hopf Galois property behaves according to the sixteen different Galois types \cite{CRV1}.
The results for $n=6$ are described in Section 4 below.

The search for new (and small) examples of Hopf Galois extensions which are not almost classically Galois or  extensions with small Galois group
not being Hopf Galois extensions led us to the study  of intermediate extensions, namely fields $F$ such that
$K\subset F\subset \wk$. If $K/k$ is Hopf Galois, since $\wk/k$ is Galois, we are dealing with extensions within
Hopf Galois extensions. If the starting extension $K/k$ is not Hopf Galois, we wonder how far we should go to achieve the Hopf Galois property.
None of these questions makes sense for the classical Galois property and are specific to the broader  context of Hopf Galois property.
In all the small degree cases we studied, we found out that when $K/k$ is already Hopf Galois, all the intermediate extensions $F/k$ are also Hopf Galois. See Section 5 for a summary of our results for $n=4,5,6.$
But in general, this is not always the case, as we prove in \cite{CRV1},
where we characterize the Hopf Galois property for intermediate extensions.

The group-theoretical description of a Hopf Galois extension  also provides an explicit description of the corresponding Hopf algebra: from $N$ we obtain the Hopf algebra $H=\wk[N]^G$ of $G$-fixed points in the group algebra $\wk[N]$, where $G$ acts on $\wk$ by field automorphisms and on
$N$ by conjugation inside $S_n$. This $H$ is a $\wk$-form of $\wk[N]$, that is $H \otimes \wk \simeq \wk[N]$.
As an example, in subsection \ref{root32} we describe in the above way, namely via descent, a Hopf algebra for the extension $\Q(\root 3 \of 2)/\Q$.
Analogous computations could be done for each of the Hopf Galois extensions considered in our work, in order to determine the attached Hopf algebra.

We complete this introduction with a review of some concepts involved in the definition and characterization of Hopf Galois extensions, as well as the
explicit examples we have mentioned before. In the following sections we address small degree extensions
and intermediate extensions.

\subsection{Normal Complements}

Our first check in a separable field extension is on the almost classically Galois property, where we should look for a normal complement of a certain
subgroup. In the case of extensions $K/k$ of degree 4, there is nothing else to be done: in each case, if $\wk$ is the normal closure and
$G'=\Gal(\wk/K)$, then $G'$ has normal complement $N$ in $G=\Gal(\wk/k)$.

\bigskip

\centerline{\bf Degree 4}

\begin{center}
\begin{tabular}{l | c | l|l }
$G$ & Size & $K/k$& $G',N$\\
\hline
\hline
$C_4$ & 4  & Galois &$1,G$\\
$V_4$ & 4  &  Galois&$1,G$\\
$D_{2\cdot 4}$ & 8  & almost classically Galois&$C_2,C_4$\\
$A_4$  &12  & almost classically Galois&$C_3,C_2\times C_2$\\
$S_4$  &24  & almost classically Galois&$S_3,C_2\times C_2$\\
\end{tabular}
\end{center}

\bigskip

We include here a couple of generic results, just to remark that
the family of Frobenius groups is the best suited for this
kind of considerations.

\begin{lemma}
Let us consider  a dihedral group
$$D_{2n}=\langle s,r| s^2=1,\ r^n=1,\ sr=r^{-1}s\rangle$$
and a subgroup $G'$ of order $2$. If $G'$ is not normal, then the cyclic subgroup  $N=\langle r\rangle$ is a normal complement of $G'$.
\end{lemma}
\begin{proof}The subgroups of $D_{2n}$ of order 2 which are not normal are
$\langle r^is\rangle$, with $0\le i\le n-1$, and we have $\langle r^is\rangle\cap \langle r\rangle=1$ for all $i$.
\end{proof}
In fact, dihedral groups of order
$2n$, with $n$ odd, are Frobenius groups with complement of order $2$.
A Frobenius group is a transitive permutation group on a finite set, such that no non-trivial element fixes more than one point and some non-trivial element fixes a point. The Frobenius complement is the stabilizer of a point. Elements in no stabilizer
together with the identity element form a normal subgroup called the Frobenius kernel. The Frobenius group is the semidirect product of these two subgroups.
\begin{lemma}
Let $F$ be a Frobenius group. If $N$ is the Frobenius kernel and $G'$ is the Frobenius complement, then $N$ is a normal complement of $G'$ in $F$. Hence if $\widetilde{K}/k$ is a Galois extension with group $F$ and $K$ is the subfield of $\widetilde{K}$ fixed by $G'$, then $K/k$ is almost classically Galois.
\end{lemma}

The preceding lemma will be applied in Section 2 to the Frobenius groups of orders 20, 21, 42, 55 and 110.

In many small degree cases, we have a unique conjugacy class of transitive subgroups of $S_n$ isomorphic to $G$. Then we can work
in $\Hol(N)$ modulo isomorphism, which is usually much easier.

\begin{prop}
Let $K/k$ be a separable extension of degree $n$ and let $\wk/k$ be its Galois closure, $G=\Gal(\wk/k)$,
and $G'=\Gal(\wk/K)$. Assume that all transitive subgroups of $S_n$ isomorphic to $G$ are in the same conjugacy class.

Then, $K/k$ is a Hopf Galois extension if and only if there exists a regular subgroup $N$ of $S_n$ such that $\Hol(N)$
has a transitive subgroup $G_1$ isomorphic to $G$.
\end{prop}
\begin{proof} Let $N$ be a regular subgroup of $S_n$ such that $\Hol(N)$
has a transitive subgroup $G_1$ isomorphic to $G$. We consider the embedding $\lambda: G\hookrightarrow S_n$. For some $\sigma\in S_n$ we have
$$
\lambda(G)=\sigma G_1 \sigma^{-1}\subset \sigma\Hol(N)\sigma^{-1}=\Hol(\sigma N\sigma^{-1}),
$$
and $\sigma N\sigma^{-1}$ is a regular subgroup of $S_n$.
\end{proof}

We list here some examples of holomorphs, many of them used in the small degree computations.
In any case, for reasonably small values of $n$, we can count on a software system like Magma to perform
explicit computations:
\begin{center}
\begin{tabular}{c|ccccccc}
$N$ & $C_2$& $C_3$& $C_4$& $C_2\times C_2$& $C_5$& $C_6$& $S_3$\\
\hline
$\Hol(N)$ & $C_2$& $S_3$& $D_{2\cdot 4}$& $S_4$& $F_{20}$& $D_{2\cdot 6}$& $S_3\times S_3$\\
\end{tabular}
\end{center}
where $F_{20}$ denotes the Frobenius group of order 20.

Although the holomorph of a group of order $n$ is smaller than $S_n$, it can also be a rather big group. For example,
$\Hol(C_2\times C_2\times C_2)$ has order 1344. Since $\Aut(C_2\times C_2\times C_2)\simeq \GL(3,2)=\SL(3,2)$, this
holomorph has a  simple subgroup of order 168 and it is not solvable. For the easiest group families the sizes of the holomorphs
are easily computed:
\begin{itemize}
\item $\Aut(C_n)\simeq C_{\varphi(n)}$, therefore $\Hol(C_n)$ is solvable and has order $n\varphi(n)$;
\item $\Aut(D_{2n})\simeq \operatorname{Aff}(\Z/n\Z)=\{ax+b \mid \gcd(a,n)=1\}$  and has order $n\varphi(n)$. Therefore,
$\Hol(D_{2n})$ has order $2n^2\varphi(n)$.
\end{itemize}

In both statements, $\varphi$ denotes the Euler function.


\subsection{Hopf algebras via descent}\label{root32}

Let us see how in practice from the regular group $N$ in Greither and Pareigis theorem we recover the Hopf algebra and the Hopf action appearing in the definition
of Hopf Galois structure.

We consider the extension $\Q(\root 3 \of 2)/\Q$. If we
denote $\alpha=\root 3\of 2$ and $\omega\in\bar\Q$ a primitive cubic root of unity, then $K=\Q(\alpha)$, its normal closure is $\widetilde{K}=\Q( \omega, \alpha)$ and $\{ 1, \alpha, \alpha^2, \omega,  \omega \alpha, \omega \alpha^2 \}$ is a basis for $\tilde K/\Q$. The Galois group is
$G \simeq S_3=\langle\tau, \sigma\rangle=\{\Id, \tau, \sigma, \sigma^2, \tau \sigma, \tau \sigma^2 \},$
with
$$\begin{array}{rl}
\tau:& \tilde K \rightarrow \tilde K \\
& \omega \mapsto \omega^2 \\
& \alpha \mapsto \alpha
\end{array}
\qquad \qquad
\begin{array}{rl}
\sigma:& \tilde K \rightarrow \tilde K \\
& \omega \mapsto \omega \\
& \alpha \mapsto \omega \alpha
\end{array}
$$
The Galois group $G'=\Gal(\tilde K/K)$ is the subgroup $G'=\langle \tau\rangle \simeq C_2$.
A left transversal for $S=G/G'$ is $x_1=\Id$, $x_2=\sigma$ and $x_3=\sigma^2$.
The left action of $G$ on $G/G'$ gives $\lambda: G \hookrightarrow B=\Sym(S) \simeq S_3$.
Therefore, $\lambda(G')$ is the stabilizer of $G'=\Id\, G'\in S$.  Using the above numbering for cosets, we obtain
$\lambda(\tau)=(2,3)$ and $\lambda(\sigma)=(1,2,3)$.

This extension $K/\Q$ is almost classically Galois: the subgroup $N=\langle \sigma\rangle$ is a normal
complement of  $G'$ in $G$. We identify $N$, $G'$ and $G$ with their images in $S_3$ under $\lambda$.
We consider
$$\tilde K[N]=\{ u_0 \,\Id+u_1 \,\sigma+ u_2 \,\sigma^2 \mid u_i\in \tilde K   \}.$$
and we look for the elements which are fixed under the $G$-action: action on $\tilde K$ as described above by
field automorphisms and action on $N$ by conjugation:
$$\begin{array}{lll}
\empty^\tau\! \Id=\Id, \quad &   \empty^\tau\! \sigma=\tau \sigma \tau=\sigma^2, \quad &  \empty^\tau\! \sigma^2=\tau \sigma^2 \tau=\sigma, \\
\empty^\sigma\! \Id=\Id,  & \empty^\sigma\! \sigma=\sigma \sigma \sigma^{-1}=\sigma, & \empty^\sigma\! \sigma^2=\sigma^2.\\
\end{array}$$
Therefore,
$$\begin{array}{l}
\tau(u_0 \Id+u_1 \sigma+ u_2 \sigma^2)=\tau(u_0)\Id+ \tau(u_1)  \sigma^2+ \tau(u_2) \sigma \\
\sigma(u_0 \Id+u_1 \sigma+ u_2 \sigma^2)=\sigma(u_0)\Id+ \sigma(u_1)  \sigma+ \sigma(u_2) \sigma^2,
\end{array}$$
and an element of $\tilde K[N]$ is fixed by  $G$ if, and only if,
$$\begin{array}{ll}
\tau(u_0)=u_0 \quad \tau(u_1)=u_2 \quad \tau(u_2)=u_1 \\
\sigma(u_0)=u_0 \quad  \sigma(u_1)=u_1 \quad \sigma(u_2)=u_2.
\end{array}
$$
This gives
$$
u_0 \in k \quad\mbox{ and }\quad u_1, u_2 \in \tilde K^{\langle \sigma\rangle}=k(\omega),
$$
and for $u_1=a+ b \omega$, with $a,b \in \Q$, we have $u_2=\tau(u_1)=a+b \omega^2$.
Putting all together, the corresponding  Hopf algebra for $K/\Q$ is
$$ \begin{array}{rcl}
H&=&\tilde K[N]^G=\{ u_0 \Id+ (a+b \omega) \sigma+(a+b \omega^2) \sigma^2 \mid u_0,a,b \in \Q  \}\\
&=&\langle \Id, \sigma+\sigma^2, \omega \sigma+ \omega^2 \sigma^2\rangle_{\Q}.\end{array}$$
This algebra is described in \cite{GP} as $\Q[c,s]/(3s^2+c^2-1,(2c+1)s,(2c+1)(c-1)).$

The Hopf action  $\mu: H \rightarrow \End_k(K)$  is given by $\mu(h)(x)=h \cdot x$. Explicitly,
if $h=h_0 \Id+ h_1(\sigma+\sigma^2)+h_2(\omega \sigma+ \omega^2 \sigma^2) \in H$ and
$x=a_0+a_1 \alpha+a_2 \alpha^2 \in K$,
then
$$\begin{array}{ll}
h \cdot x & = h_0 x+ h_1(\sigma(x)+ \sigma^2(x))+h_2(\omega \sigma(x)+ \omega^2\sigma^2(x))= \\
          &=a_0 (h_0+2 h_1-h_2)+a_1  (h_0-h_1-h_2)\,\alpha+a_2 (h_0-h_1+2 h_2)\,\alpha^2.
\end{array}$$


\section{Classifying Hopf Galois structures}\label{count}

The group theoretic description of Hopf Galois extensions given by Greither and Pareigis
showed that there exist non-trivial Hopf Galois structures for separable field extensions
and opened the question of counting and classifying Hopf Galois structures
for a given separable field extension.

The fact that $K/k$ is classically Galois, or almost classically Galois, does not mean that the strong form of the fundamental theorem of Galois theory holds for all Hopf Galois structures on $K/k$. The first such example is due to Greither and Pareigis for $K/k$ any classical Galois extension with non-abelian Galois group $G$. In this case, there is another way than $\lambda$ to embed $G$ inside $\Sym(G)$:
$$
\rho: G\hookrightarrow \Sym(G/G')=\Sym(G)
$$
with $\rho(\sigma)(\tau)=\tau\sigma^{-1}$. This is a regular embedding and $\rho(G)$ is normalized by $\lambda(G)$.
Therefore, if $K/k$ is a non-abelian Galois extension, then $\lambda(G)\ne \rho(G)$ and there are at least two different Hopf Galois structures, corresponding to the regular
subgroups $N_1=\lambda(G)$ and $N_2=\rho(G)$.
In fact, $\rho(G)$ correspond to the classical action of $G$ on $K$ and by considering $\lambda(G)$ one
gets the following result.

\begin{theorem}[\cite{GP} Theorem 5.3]\label{gal}
Any Galois extension $K/k$ can be endowed with an $H$-Galois structure such that there is a canonical bijection between
sub-Hopf algebras of $H$ and normal intermediate fields $k\subseteq E\subseteq K$.
\end{theorem}

In subsection \ref{HGC4} we show that non-unicity of Hopf Galois structures can be found already in a classical Galois extension $K/k$ with Galois group $G=C_2\times C_2$, the Klein 4-group.
In that example we see that the four different Hopf Galois structures provide four different images $Im{\mathcal F}_H$ inside the
lattice of subfields of $K/k$. And we only get surjectivity in the classical case. We have already mentioned that
we can have surjectivity for ${\mathcal F}$ for non-classical Hopf Galois structures: for example the almost classically Galois extensions.
The example in subsection \ref{ImF}  shows that we can also have different Hopf Galois structures with the same image for ${\mathcal F}$.

\subsection{Counting Hopf Galois structures}
The equivalent condition to the Hopf Galois property for separable field extensions given in Theorem \ref{defHG} gives a bijection between isomorphism classes of Hopf Galois structures on $K/k$ and regular subgroups $N$ of $S_n$ normalized by $G$ (see \cite{GP} Theorem 3.1). The following theorem makes more precise the  relationship between $G$ and $N$.

\begin{theorem}[Byott \cite{Byott96} Proposition 1] \label{Byott}Let $G$ be a finite group, $G'\subset G$ a subgroup and $\lambda:G\to \Sym(G/G')$ the homomorphism corresponding to
the action of $G$ on the left cosets $G/G'$. Let $N$ be a group of order $[G:G']$ with identity element $e_N$.
Then there is a bijection between
$$
{\mathcal N}=\{\alpha:N\hookrightarrow \Sym(G/G') \mbox{ such that }\alpha (N)\mbox{ is regular}\}
$$
and
$$
{\mathcal G}=\{\beta:G\hookrightarrow \Sym(N) \mbox{ such that }\beta (G')\mbox{ is the stabilizer of } e_N\}.
$$
Under this bijection, if $\alpha\in {\mathcal N}$ corresponds to $\beta\in {\mathcal G}$, then
$\alpha(N)$ is normalized by $\lambda(G)$ if and only if
$\beta(G)$ is contained in $\Hol(N)$.
\end{theorem}

To count Hopf Galois structures on a separable extension $K/k$ with normal closure $\wk$, and $G=\Gal(\wk/k)$, $G'=\Gal(\wk/K)$, we
seek regular subgroups of $\Sym(G/G')$ normalized by $\lambda(G)$. This counting is made more treatable by the following proposition, which is
a corollary of Theorem \ref{Byott} (see the sentence before Proposition 1 in \cite{Byott96}).

\begin{prop}
Let $K/k$ be a separable extension with normal closure $\wk$, and $G=\Gal(\wk/k)$, $G'=\Gal(\wk/K)$. Let
${\mathcal S}$ be the set of isomorphism classes of groups $N$ with $|N|=[G:G']$. The number
of Hopf Galois structures on $K/k$ is
$$
s(G,G')=\sum_{\{N\}\in{\mathcal S}} e(G,N)
$$
 $e(G,N)$ being the cardinality of the set of equivalence classes of embeddings $\beta$ of $G$ into $\Hol(N)$ such
that $\beta(G')$ is the stabilizer of $e_N$, modulo conjugation by elements of $\Aut(N)\subset \Hol(N)$.
\end{prop}

The set  $\mathcal S$ parametrizes the types and $e(G,N)$ is the number of Hopf Galois structures \emph{of type $N$}: Hopf
Galois structures on $K/k$ with $k-$Hopf algebras $H$ such that $\wk\otimes H\simeq \wk[N]$.
Although the structure of the Hopf algebras acting on $K/k$ depends on
the extension, the question of how many Hopf Galois structures there are on a given $K/k$ depends only on
$G$ and $G'$. And much remains unknown on this question.

The first case to consider is $G'=1$, namely the case of classical Galois extensions, and a natural question is to characterize
when the classical structure is the unique one.  As it was mentioned above, all non-abelian
Galois extensions are examples of non-unicity of Hopf Galois structures.
We find unicity in a narrow class of abelian extensions.

\begin{prop}[Byott \cite{Byott96} Theorem 1]
A Galois extension $K/k$ with Galois group $G$ has a unique Hopf Galois structure if, and
only if, $n=|G|$ is a Burnside number, that is, $(n,\varphi(n))=1$, where  $\varphi$ denotes the Euler function. In particular,
all these extensions are cyclic.
\end{prop}

We also find unicity of Hopf Galois structure in the case of extensions of prime degree.

\begin{prop}[Childs \cite{Childs1} Theorem 2] If $K/k$ is a separable field extension of prime degree, then
$$
K/k \mbox{ is Hopf -Galois } \iff \Gal(\wk/k) \mbox{ is solvable}.
$$
Besides, in this case $K/k$ is almost classically Galois and has a unique Hopf Galois structure.
\end{prop}

Therefore, we can completely classify the extensions of small prime degree.

\vspace{1cm}

\centerline{\bf Degree 3. $\mathbf{\mbox{Hol}(C_3)=S_3}$}
\begin{center}
\begin{tabular}{c|c|l|l}
$\Gal(\wk/k)$& Size&$K/k$ & $N=3$-Sylow of $G$\\
\hline
\hline
$C_3\simeq A_3 $ &3& Galois & $N=G$\\ [1ex]
$D_{2\cdot 3}\simeq S_3$ &6& almost classically Galois & $G'=\langle s\rangle$, $N=\langle r\rangle$\\ [1ex]
\end{tabular}
\end{center}

\bigskip

\centerline{\bf Degree 5. $\mathbf{\mbox{Hol}(C_5)=F_{20}}$}

\begin{center}
\begin{tabular}{c|c|l|l}
$\Gal(\wk/k)$& Size&$K/k$ & $N=5$-Sylow of $G$\\
\hline
\hline
$C_5 $ &5& Galois & $N=G$\\ [1ex]
$D_{2\cdot 5}$ &10& almost classically Galois & $G'=\langle s\rangle$, $N=\langle r\rangle$\\ [1ex]
$F_{20}$ &20& almost classically Galois& $G'=$ Frobenius complement \\
                      &  && $N\ =$ Frobenius kernel \\ [1ex]
$A_5$& 60&not Hopf Galois&\\
$S_5$ &120
& not Hopf Galois&\\
\end{tabular}
\end{center}

\bigskip

\centerline{\bf Degree 7. $\mathbf{\mbox{Hol}(C_7)=F_{42}}$}

\begin{center}
\begin{tabular}{c|c|l|l}
$\Gal(\wk/k)$& Size&$K/k$ & $N=7$-Sylow of $G$\\
\hline
\hline
$C_7 $ &7& Galois & $N=G$\\ [1ex]
$D_{2\cdot 7}$ &14& almost classically Galois & $G'=\langle s\rangle$, $N=\langle r\rangle$\\ [1ex]
$F_{21}$ &21& almost classically Galois& $G'=$ Frobenius complement \\
                      &  && $N\ =$ Frobenius kernel \\ [1ex]
$F_{42}$ &42& almost classically Galois & $G'=$ Frobenius complement \\
                      &  && $N\ =$ Frobenius kernel \\ [1ex]
$PSL(2, 7)$& 168&not Hopf Galois&\\
$A_7$& 2520&not Hopf Galois&\\
$S_7$ &5040
& not Hopf Galois&\\
\end{tabular}
\end{center}

\vspace{1cm}

\centerline{\bf Degree 11. $\mathbf{\mbox{Hol}(C_{11})=F_{110}}$}

\begin{center}
\begin{tabular}{c|l|l}
$\Gal(\wk/k)$& $K/k$ & $N=11$-Sylow of $G$\\
\hline
\hline
$C_{11}$ & Galois & $N=G$\\ [1ex]
$D _{2\cdot 11}$ & almost classically Galois & $G'=\langle s\rangle$, $N=\langle r\rangle$\\ [1ex]
$F_{55}$ & almost classically Galois& $G'=$ Frobenius complement \\
                      &  & $N\ =$ Frobenius kernel \\ [1ex]
$F_{110}$ & almost classically Galois & $G'=$ Frobenius complement \\
                      &  & $N\ =$ Frobenius kernel \\ [1ex]
$PSL(2, 11)$ & not Hopf Galois&\\
$M_{11}$& not Hopf Galois&\\
$A_{11}$& not Hopf Galois&\\
$S_{11}$ & not Hopf Galois&\\
\end{tabular}
\end{center}

\vspace{1cm}

When we consider a Hopf Galois extension of degree $n$, with $n$ a Burnside number, we know by \cite{Byott96} that the Galois group of
its normal closure must be solvable. 
But the converse is not true: in degree 15 there are extensions with Galois group of order 150, which is solvable; but such an extension
cannot be Hopf Galois, since $\Hol(C_{15})$ has order 120.
\bigskip

Back to the case of counting Hopf Galois structures, Kohl realizes a complete counting for a family of extensions of prime power degree.

\begin{theorem}[Kohl \cite{Kohl}]
Let $p$ be an odd prime, $n$ a positive integer and $k$ a field of characteristic $0$. Let $K=k(\alpha)$, where $X^{p^n}-a$ is the minimal polynomial of $\alpha$ over $k$, and let $r$ denote the largest integer between $0$ and $n$ such that $K$ contains a primitive $p^r$th root of unity.
\begin{enumerate}[(1)]
\item For $r<n$, there are $p^r$ Hopf Galois structures on $K/k$ for which the associated group $N$ is cyclic of order $p^n$. Of these, exactly $p^{min(r,n-r)}$ are almost classically Galois.
\item For $r=n$ (i.e. when $K/k$ is a cyclic extension of order $p^n$), there are $p^{n-1}$ Hopf Galois structures for which $N$ is cyclic of order $p^n$.  Of these, exactly one is almost classically Galois.
\end{enumerate}

\noindent In both cases, these are the only possible Hopf Galois structures on $K/k$.
\end{theorem}

For classical Galois extensions, the excluded case $p = 2$ is treated in \cite{Byott07}, where it is proved that a cyclic Galois extension of degree
$2^n$, $n\ge 3$,  admits  $3 \cdot 2^{n-2}$ Hopf Galois structures. They are equally distributed among the
three possible types: $N$ can be the cyclic group $C_{2^n}$ , the dihedral group $D_{2^n}$,
or the generalized quaternion group $Q_{2^n}$ and  the almost classically Galois structures are of cyclic type.
Byott also proves:
\begin{itemize}
\item[-] for a Galois extension of degree $p^2$, there are exactly $p$ distinct Hopf Galois structures if the Galois group is cyclic and $p^2$ if the Galois group is
elementary abelian \cite{Byott96}
\item[-] for a Galois extension of degree $pq$, where $p, q$ are primes and $p \equiv 1 \bmod q$, there are $2q - 1$, respectively $2+p(2q - 3)$, Hopf
Galois structures when the extension is cyclic, respectively non-abelian \cite{Byott04}. For example, this gives 5 distinct
Hopf Galois structures for a Galois extension with Galois group $S_3$.
\item[-] for a Galois extension whose Galois group is a non-abelian simple group there are exactly two different Hopf Galois structures \cite{Byott042}.
\end{itemize}

Other cases where the classification has been addressed are
Galois extensions of order $4p$, where $p$ is an odd prime \cite{Kohl07},
Galois extensions with groups  $G$ that are semidirect products of cyclic groups and have trivial centers \cite{Ch-Corr} or
Galois extensions or order $mp$, where $p$ is prime and $m < p$ \cite{Ko13}. A non-unicity result for abelian extensions is given in \cite{By-Chi}, namely that every finite Galois field extension with abelian group of even order $>4$ admits a Hopf Galois structure for which the associated group $N$ is non-abelian.


\subsection{Non-unicity of Hopf Galois structures in a degree 4 abelian extension}\label{HGC4}

The non-unicity of Hopf Galois structures appears already  for Galois extensions of degree 4.
If the Galois group is cyclic, there are 2 distinct structures, and
if the Galois group is elementary abelian there are 4 distinct structures. Let us show this case in detail.

Let $k$ be a field of characteristic $\ne 2$ and $K/ k$ a Galois extension with Galois group $G$ isomorphic to the Klein group:
$$G \simeq C_2 \times C_2=\langle\sigma, \tau \rangle=\{\Id, \sigma, \tau, \sigma \tau= \tau \sigma \}.$$
We can write $K=k(\sqrt{a}, \sqrt{b})$ with $a,b,ab\in k^*\setminus k^{*2}$. Then,  $\{1, \sqrt{a}, \sqrt{b}, \sqrt{ab}\}$ is a $k-$basis of $K$ and the Galois action
is given by
$$\begin{array}{rcrcrc}
\sigma:& K \rightarrow K   & \qquad \tau:& K \rightarrow K &  \qquad \sigma \tau:& K \rightarrow K \\
& \sqrt{a} \mapsto -\sqrt{a} & & \sqrt{a} \mapsto \sqrt{a} && \sqrt{a} \mapsto -\sqrt{a} \\
& \sqrt{b} \mapsto \sqrt{b} & & \sqrt{b} \mapsto -\sqrt{b} && \sqrt{b} \mapsto -\sqrt{b}.
\end{array}$$
If we consider the regular representation $\lambda: G\hookrightarrow \Sym(G)$ and we take, for example, the enumeration
$x_1=\Id,\ x_2=\sigma,\ x_3= \tau, \ x_4=\sigma \tau$, then
$ \lambda(\sigma)=(1,2)(3,4)$ and $\lambda(\tau)=(1,3)(2,4)$ in $S_4$. This identifies $G$ with a regular subgroup of $S_4$, the subgroup $V_4$ formed by
the identity and the three products of two disjoint transpositions.

Now, to look for the different Hopf Galois structures of $K/k$ we should look for regular subgroups $N$ of $S_4$ normalized by $V_4$.
Taking $N=V_4$ we get the classical Galois structure, which corresponds to the group algebra $k[G]$. But the three
cyclic subgroups of order 4
$$
N_1=\langle(1,2,3,4)\rangle, \qquad N_2=\langle(1,3,2,4)\rangle, \qquad N_3=\langle(1,4,2,3)\rangle
$$
are also normalized by $V_4$. Altogether, these are the 4 different Hopf Galois structures for $K/k$. We describe the
Hopf algebra corresponding to $N_1$ and  the correspondence $\mathcal F$ of the main theorem in this case.

If we let  $g_1=(1,2,3,4)$, then the Hopf algebra is
$$H_{1}=K[N_1]^G=\{z=a_0 \Id+a_1 g_1+a_2 g_1^2+a_3 g_1^3 \in K[N_1] \mid  \, \empty^g\! z =z, \forall g \in G \}.$$
Since,
$$\begin{array}{l}
 \empty^\sigma\! g_1=\lambda(\sigma) g_1 \lambda(\sigma^{-1})=(1,2)(3,4)(1,2,3,4)(1,2)(3,4)=(1,4,3,2)=g_1^3,\\
 \empty^\tau\! g_1=\lambda(\tau) g_1 \lambda(\tau^{-1})=(1,3)(2,4)(1,2,3,4)(1,3)(2,4)=(1,2,3,4)=g_1,\\
 \end{array}$$
we obtain $z \in H_{1}$ if, and only if,
$$
\sigma(a_0)=a_0, \, \,  \sigma(a_1)=a_3, \, \,  \sigma(a_2)=a_2, \, \,  \sigma(a_3)=a_1, \, \,  \tau(a_i)=a_i.$$
This gives $a_0,a_2 \in k$ and $a_1 \in k(\sqrt{a})$. Furthermore, if $a_1=x_0+x_1 \sqrt{a}$, \,  $x_i \in k$, then  $a_3=\sigma(a_1)=x_0-x_1 \sqrt{a}$.
Therefore, the Hopf algebra is
$$H_{1}=K[N_1]^G=\langle 1, \, \,  g_1^2,\, \,  g_1+g_1^3, \, \, \sqrt{a}(g_1-g_1^3)\,\rangle.$$
Since the Hopf action $\mu_{1}: H_{1} \rightarrow \End_k(K)$ is given by
$$ (\sum_{n \in N_1}c_n n)\, x=\sum_{n \in N_1}c_n(n^{-1}(1_G))\, x.$$
and
$
\Id^{-1}(1_G)=\Id, \,  \, g_1^{-1}(1_G)= \sigma \tau, \, \, (g_1^2)^{-1}(1_G)= \tau,\, \, (g_1^3)^{-1}(1_G)= \sigma,$
we have the Hopf action $$\mu_1(a_0+ a_1 g_1+a_2 g_1^2+a_3 g_1^3)(x)= a_0 x+ a_1 \sigma \tau(x)+ a_2 \tau(x)+a_3 \sigma(x).$$

Corresponding to the unique subgroup $\langle g_1^2\rangle $ of $N_1$, the algebra $H_1$ has the unique sub-Hopf algebra,
 $F=k[\langle 1, g_1^2\rangle]=K[\langle g_1^2\rangle]^G$, which is 2-dimensional. The fixed subfield is
$$ \begin{array}{rl}
K^{F}&= \{ x=x_0+x_1 \sqrt{a}+x_2 \sqrt{b}+x_3 \sqrt{ab} \in K \mid  \mu_{1}(h)(x)=\varepsilon (h) x , \, \forall h \in F \}=\\
   &= \{ x \in K \mid  \mu_{1}(g_1^3)(x)= \tau(x)= x \}\\
   &= \{ x \in K \mid  x_0+x_1 \sqrt{a}-x_2 \sqrt{b}-x_3 \sqrt{ab}= x \}.\end{array}$$
Therefore,  $x_2=x_3=0$ and $x=x_0+x_1 \sqrt{a}$. Namely, $K^{F}=k(\sqrt{a}).$

In this example, the Hopf Galois structure provided by $N_1$ is not classical or almost classically Galois and the main theorem
does not hold in its strong form, the sub-Hopf algebras only provide a portion of the subfield lattice: the image of ${\mathcal F}_{H_1}$ is
$$\begin{array}{c}
 K\\
 \mid \\
 K^F=k(\sqrt{a}) \\
 \mid \\
 k
\end{array}
$$
The remaining portions of the subfield lattice are obtained analogously through $N_2$ and $N_3$. And, of course, the classical structure
reflects the whole lattice.


\section{The lattice of sub-Hopf algebras and the main theorem}

The main theorem concerns sub-Hopf algebras of the Hopf algebra attached to the Hopf Galois structure under consideration.
Since in the separable case all these algebras are forms of a certain group algebra, the group algebra $\wk[G]$,
we start by recalling that the sub-Hopf algebras of the group algebra correspond to subgroups of $G$.

\begin{prop}(\cite{CRV2})
Let $k$ be a field and $G$ a finite group. The sub-Hopf algebras of  $k[G]$ are the group algebras
 $k[T]$, with $T$ a subgroup of  $G$.
\end{prop}

Let us remark that  the set of group-like elements of a group algebra $k[T]$ is precisely $T$ and therefore from different subgroups we obtain different
subalgebras.
Therefore, with the notation we have been using for separable Hopf Galois extensions, for a subgroup $N$ providing
a Hopf Galois structure in $K/k$,
$$
\begin{array}{rcl}
\{N' \mbox{ subgroup of }N \} &\longrightarrow& \{H'\subseteq \wk[N] \mbox{  sub-Hopf algebra}\}\\
N'&\to& \wk[N']
\end{array}
$$
is a one-to-one correspondence. If we consider $H=\wk[N]^G$, the Hopf algebra obtained from $\wk[N]$ via descent, we can define
$$
\begin{array}{rcl}
\{N' \mbox{ subgroup of }N \} &\longrightarrow& \{H'\subseteq H \mbox{  sub-Hopf algebra}\}\\
N'&\to& \wk[N']^G
\end{array}
$$
where $\wk[N']^G=\wk[N']\cap \wk[N]^G$. If $N'$ is stable under the conjugacy action of $\lambda(G)$ in $N$, then $G$ acts in $\wk[N']$ and we are considering the fixed points.

The above map is surjective because for a given $H'$ by extension of  scalars we obtain a sub-Hopf algebra of $\wk[N]$ and via descent we
recover $H'$.
We do not have injectivity in general. For example, in some cases (see the next subsection) we get $\wk[N']^G=k=\wk[\{1\}]^G$ for a non-trivial $N'$.
But if all the subgroups of $N$ are stable under conjugation by elements of $\lambda(G)$ (as happened in the examples of the previous subsection) then
we have injectivity.
In fact, since $N'$ and the stable subgroup $\bigcap\limits_{\sigma \in G} \lambda(\sigma) N'\lambda(\sigma)^{-1}$ give rise to the same sub-Hopf algebra, we have that
sub-Hopf algebras of $H$ are in bijection with subgroups of $N$ stable under the action of $\lambda(G)$ (see \cite{CRV2}).
Therefore, the  main theorem admits also a group-theoretical reformulation:
$$
\begin{array}{rcl}
{\mathcal F}_N:\{\mbox{Subgroups }N' \subseteq N \mbox{ stable under }\lambda(G)\} &\longrightarrow&\{\mbox{Fields }E\mid k\subseteq E\subseteq K\}\\
N'&\to &K^{\wk[N']^G}
\end{array}
$$
is injective and inclusion reversing. For every Hopf Galois structure on $K/k$ we may define a map $\mathcal F$ and all of them have image in the same
set, the lattice of subfields of $K/k$

\subsection{The image of  ${\mathcal F}$}\label{ImF}

In the previous section we have seen that in the case of biquadratic extensions of  Hopf Galois type $N\simeq C_4$ we do not get surjectivity, that is,
the main theorem does not hold in its strong form. In that example, the four distinct Hopf Galois structures gave rise to different images of the
corresponding $\mathcal F$.  We have a similar situation when we analyze the five distinct Hopf Galois structures of a Galois extension with Galois
group $S_3$. Since it is non-abelian, we have the classical structure giving surjective $\mathcal F$ and another structure of type $S_3$ such that the image
of $\mathcal F$ is the set of normal intermediate fields, that is $k\subset F \subset K$, with $[F:k]=2$. The remaining three Hopf Galois structures  are
of cyclic type and the images of the corresponding $\mathcal F$ describe each  of the cubic subfields (see \cite{CRV2}, where the whole family of dihedral groups $D_{2p}$, $p$ an odd prime, is treated).

But we can have different Hopf Galois structures giving the same image for $\mathcal F$. For example, if both structures are almost classically Galois then
both $\mathcal F$ are surjective. This is the case when
we take a characteristic zero field $k$ such that $\ii=\sqrt{-1} \notin k$ and we consider $\al\in\bar k$  with minimal polynomial $x^4-a\in k$
Then, $K=k(\alpha)/k$ is a degree 4 separable extension with normal closure $\widetilde{K}=k( \al,\imath)$ and Galois group
$G$ isomorphic to the dihedral group $D_{2\cdot 4}=\langle s,r| s^2=1,\ r^4=1,\ sr=r^{-1}s\rangle$.
We can enumerate the roots of $x^4-a$ so that
$G$ is identified as a transitive subgroup of $S_4$ via $r=(1,2,3,4)$ and $s=(2,4)$.
The group $G'=\Gal(\wk/K)$ is the stabilizer of a point, therefore  $G'=\langle s\rangle $ (modulo conjugation in $G$, which corresponds to
rename the roots of the polynomial, or renumbering them).
It has two normal complements in $G$,
$$N_1=\langle r\rangle\simeq C_4 \mbox{ and }
N_2=\langle r^2, s r \rangle \simeq C_2 \times C_2,$$
so we have two different almost classically Galois structures. Since the main theorem holds in its strong form,
we get two different Hopf algebras $H_1$ and $H_2$ such that $\mathrm{Im\,}{\mathcal F}_{H_1}=\mathrm{Im\,}{\mathcal F}_{H_2}$:
$$\begin{array}{rcl} H_1&=&\{ \la_0 \Id+ (a+b \ii) r+ \la_2 r^2+(a-b \ii) r^3 \mid \la_0, \la_2,a,b \in k  \}
\\&=&<Id, r+r^3, \ii(r-r^3), r^2>_k\end{array}$$
and
$$\begin{array}{rcl}
H_2&=&\{ \la_0 \Id+ \la_1 r^2+ (a+b \ii \al^2) sr+(a-b \ii \al^2) rs\mid \la_0, \la_1,a,b \in k  \}=\\
&= &<Id, r^2, s r+rs, \ii \al^2 (sr-rs)>_k.
\end{array}$$
As for the sub-Hopf algebras,  we only have a proper subgroup of $N_1$, and the fixed field for the sub-Hopf algebra
$F_1=\wk[\langle r^2\rangle]^G$
 is $K^{F1}=k(\alpha^2)$. For $N_2$, we have three subgroups of order 2, and a priori we can consider three sub-Hopf algebras:
$$
\wk[\langle r^2\rangle]^G, \quad \wk[\langle sr\rangle]^G\quad \mbox{ and }\quad \wk[\langle sr^3\rangle]^G.
$$
But only the first one is 2-dimensional, the other two are just $k$. The first subgroup is a normal subgroup of $G$ (stable by conjugation),
for the other two subgroups the intersection of their conjugates is trivial.

Again, one may think that almost classically Galois are too similar to classical Galois extensions and that we have to experiment
with  Hopf Galois extensions not almost classically Galois in order to get more information. In \cite{CRV2} one can see examples
of  non-almost classically Galois extensions having surjective $\mathcal F$. Therefore, the main theorem holds in its strong form for a class
of extensions wider than the class of almost classically Galois extensions and the characterization of this class remains an open problem.

\section{Hopf Galois in small degree}

In previous sections we have already seen the Hopf Galois classification for separable field extensions $K/k$ of
degree 4 and 5. Let us consider the case $[K:k]=6$, where we can see the power of the reformulation in terms of holomorphs given in Theorem \ref{hol}.

Since there are only two isomorphism classes of groups of order 6, we just have to consider
$\Hol(C_6)\simeq D_{2\cdot 6}$ and $\Hol(S_3)\simeq S_3\times S_3\simeq F_{18}:2$.
This last notation follows the naming scheme developed in \cite{CHM}.

None of these holomorphs can have a subgroup of order $24$ o bigger than $36$, none of them has a subgroup
$A_4$ and $\Hol(S_3)$ is not isomorphic to the Frobenius group $F_{36}$.
These facts rule out many possibilities of the list of transitive subgroups of $S_6$ and the remaining ones correspond to Galois or almost classically Galois
extensions. The following table shows the complete classification and more details can be found in \cite{CRV1}.

\begin{center}
\begin{tabular}{l | c | l }
$\Gal(\wk/k)$ & Size & $K/k$\\
\hline
\hline
$C_6$&6  & Galois\\
$S_3$&6 & Galois\\
$D_{2\cdot 6}$ &12 & almost classically Galois\\
$A_4$ &12 & not Hopf Galois \\
$F_{18}$&18 & almost classically Galois\\
$2A_4$ &24 & not Hopf Galois\\
$S_4(6d )$ &24 & not Hopf Galois\\
$S_4(6c)$ &24 & not Hopf Galois\\
$F_{18} : 2$&36 & almost classically Galois\\
$F_{36}$ &36 & not Hopf Galois\\
$2S_4$ &48 & not Hopf Galois\\
$A_5$ &60 & not Hopf Galois\\
$F_{36}: 2$&72 & not Hopf Galois\\
$S_5$ &120 & not Hopf Galois\\
$A_6$ &360 & not Hopf Galois\\
$S_6$ &720 & not Hopf Galois\\
\end{tabular}
\end{center}

\section{Intermediate extensions}

Now we are interested in intermediate fields $K\subset F \subset \tilde K$, since we do not question about $F/k$ being Galois
but we can question about the Hopf Galois condition  for $F/k$. The first interesting case appears already for $[K:k]=4$, where we have non-trivial intermediate fields when
$\Gal(\wk/k)=S_4$. Then, $G'=\Gal(\wk/K)$ is isomorphic to $S_3$ and we can
consider a subgroup $G''\simeq C_3$ of $\Gal(\wk/K)$ and  the fixed field $F=\tilde K^{G''}$.
Since $S_4$ has no normal subgroups of order 8,  $G''$ has no  normal complement in $S_4$ and $F/k$ is not almost classically Galois.
It is shown in \cite{CRV1} that  $F/k$ is a Hopf Galois extension of type $N=C_2\times C_2\times C_2$.
$F=\Q(\alpha,\sqrt{229})$, with $\alpha$ a root of $X^4+X+1\in\Q[x]$, is an explicit example in the smallest possible degree
of such a non-almost classically Galois Hopf Galois extension.

\subsection{Intermediate extensions in small degree}
In degree 4 all the extensions are either Galois or almost classically Galois and the only case with non-trivial intermediate fields is
the one considered above. In this case, intermediate fields of degree 12 are almost classically Galois.

\bigskip
\centerline{$\mathbf{[K:k]=4\qquad k \subset K \subset F \subset \widetilde{K}}$}
\begin{center}
\begin{tabular}{c|c | c | l }
$\Gal(\wk/k)$ &$K/k$ & $[F:k]$ & $F/k$\\
\hline
$S_4$ & Hopf Galois & 8  & Hopf Galois not \\&&& almost classically Galois \\ [1ex]
          && 12  &  almost classically Galois\\

\end{tabular}
\end{center}

\bigskip
For $[K:k]=5$ we have more variety since we have also non-Hopf Galois extensions with non-trivial intermediate
fields. Between a Hopf Galois extension and its Galois closure we find again a small degree example of  Hopf Galois
extensions which are not almost classically Galois. Between a non-Hopf Galois extension and its Galois closure, it seems
that one has to get close to this Galois closure to achieve the Hopf Galois property.

\vspace{1cm}
 \centerline{$\mathbf{[K:k]=5\qquad k \subset K
\subset F \subset \widetilde{K}}$}
\begin{center}
\noindent\begin{tabular}{c|c | c | l }
$\Gal(\wk/k)$ &$K/k$ & $[F:k]$ & $F/k$\\
\hline
$F_5$ & Hopf Galois & 10  & Hopf Galois not\\&&& almost classically Galois \\ [1ex]
 $A_5$& Not Hopf Galois& $15, 20, 30$& Not Hopf Galois \\  [1ex]
   $S_5$& Not Hopf Galois& $10, 15, 20, 30,40$& Not Hopf Galois \\
&& 60& Almost classically Galois \\
\end{tabular}
\end{center}

\bigskip

The range of cases in degree 6 is also detailed in \cite{CRV1} and we collect it in the following table.

\centerline{$\mathbf{[K:k]=6\qquad k \subset K \subset F \subset \widetilde{K}}$}
\begin{center}
\noindent\begin{tabular}{c|c | c | l }
$\Gal(\wk/k)$ &$K/k$ & $[F:k]$ & $F/k$\\
\hline
$F_{18}:2$ & Hopf Galois & $12$  & Hopf Galois not\\&&& almost classically Galois \\ [1ex]
                    & & $18$& Almost classically Galois \\  [1ex]
   $2A_4$& Not Hopf Galois& $12$& $\exists$ almost classically  Galois \\  [1ex]
$S_4(6c)$&Not Hopf Galois& 12&  Not Hopf Galois \\ [1ex]
$S_4(6d)$&Not Hopf Galois& 12&  $\exists$ almost classically  Galois \\ [1ex]
$F_{36}$&Not Hopf Galois& 12&  Not Hopf Galois \\
                   & & $18$& Hopf Galois not \\
                   & & & almost classically Galois \\  [1ex]
$2S_4$&Not Hopf Galois& 12&  Not Hopf Galois \\
                   & & $24$& $\exists$ almost classically  Galois \\ [1ex]
 $A_5$&Not Hopf Galois& $12,30$&  Not Hopf Galois \\  [1ex]
   $F_{36}:2$&Not Hopf Galois& $12,24$&  Not Hopf Galois \\
                      & & $18$& Hopf Galois not \\
                   & & & almost classically Galois \\
                  & & $36$& $\exists$ almost classically  Galois \\ [1ex]
 $S_5$&Not Hopf Galois& $12,24,30,60$&  Not Hopf Galois \\  [1ex]
$A_6$&Not Hopf Galois& $30,36,60,72$&  Not Hopf Galois \\
  & & 90,120,180& \\ [1ex]
$S_6$&Not Hopf Galois& $< 360$&  Not Hopf Galois \\
  & & 360 & $\exists$ almost classically  Galois \\
\end{tabular}
\end{center}


\subsection{Transitivity of Hopf Galois property}

When we started with a non-Hopf Galois extension $K/k$ and we study chains of subfields between $K$ and the Galois closure  $\wk$ looking
for the smallest Hopf Galois extension of $k$, we do not know how  to predict when we are going to find it.
But when we start with a Hopf Galois extension, we have a better knowledge of
what happens with the intermediate lattice of subfields of $\wk/K$.

\begin{theorem}[\cite{CRV1}]
 Let $K/k$ be a separable field extension and $\wk/k$ its Galois closure. Let $F$ be an intermediate field
  $K \subset F \subset \wk$. If $K/k$ and $F/K$ are Hopf Galois extensions, then
   $F/k$ is also a Hopf Galois extension.
\end{theorem}

In all the shown examples of Hopf Galois extensions $K/k$, intermediate fields $K\subset F\subset \tilde K$
are also Hopf Galois extensions of $k$. The condition $F/K$ Hopf Galois in the above theorem is fulfilled in these cases
because of its small degree.
But a Hopf Galois extension $K/k$ of degree 60 can be found  with an intermediate field $F$ such that $F/K$ has degree 5 and Galois closure
with Galois group $A_5$. Therefore the relative extension $F/K$ is not Hopf Galois. On the other hand, the theorem cannot be extended to
the composition of arbitrary separable field extensions, since from the composition of a quadratic and a cubic extension, both Hopf Galois,
we can obtain a sextic field having normal closure with Galois group $F_{36}$, therefore non-Hopf Galois.

\section{Further developments}

 In \cite{Childs0} Childs introduced the idea that Hopf algebras could fruitfully broaden the domain of Galois module theory, namely the
branch of algebraic number theory which studies rings of integers of Galois extensions of number fields as modules over
the integral group ring of the Galois group.

If $K/k$ is a Galois extension of algebraic number fields with Galois group $G$ then $K$ is a $k[G]$-module and the normal basis
theorem states that $K$ is a free $k[G]$-module of rank 1. If we consider local fields with valuation rings ${\mathcal O}_K$ and ${\mathcal O}_k$
we can ask if ${\mathcal O}_K$ has a normal basis as a ${\mathcal O}_k$-free module. Noether's theorem states that this is true if and only
if $K/k$ is tamely ramified.

In the attempt to address the wildly ramified extensions Leopoldt proposed to replace the group ring ${\mathcal O}_k[G]$ by a larger
order, the associated order
$$
\mathfrak A=\{\alpha\in k[G] \mid \alpha(x)\in {\mathcal O}_K  \mbox{ for all } x\in {\mathcal O}_K\}
$$
and proved that if $K/\Q_p$ is abelian then ${\mathcal O}_K$ is a free $\mathfrak A$-module of rank one.
But if one replaces $\Q_p$ by another base field or considers non-abelian extensions, then the result is not true.

Childs proves that for wildly ramified extensions, freeness is deduced from $\mathfrak A$ being a Hopf order of $k[G]$.
With this point of view, one can expand from classical Galois structure to Hopf Galois structures. If $K/k$
is Hopf Galois with Hopf algebra $H$, define the
associated order in the same way
$$
\mathfrak A_H=\{\alpha\in H \mid \alpha(x)\in {\mathcal O}_K  \mbox{ for all } x\in {\mathcal O}_K\}
$$
and the same result holds: if $\mathfrak A$ is a Hopf order of $H$, then ${\mathcal O}_K$ is a free $\mathfrak A$-module of rank one.
In \cite{Byott97} Byott gives examples of wildly ramified Galois extensions for which
${\mathcal O}_K$ is not free over $\mathfrak A_{k[G]}$ but is free over $\mathfrak A_H$ for some
Hopf algebra $H$ giving a non-classical structure. We refer to \cite{Childs} for a survey on this subject and to the work of
Truman \cite{Tru} for recent results on this subject.


\bibliographystyle{amsplain}

\vspace{1cm}

\footnotesize

\noindent Teresa Crespo, Departament d'\`Algebra i Geometria,
Universitat de Barcelona, Gran Via de les Corts Catalanes 585,
E-08007 Barcelona, Spain, e-mail: teresa.crespo@ub.edu

\vspace{0.3cm} \noindent Anna Rio, Departament de Matem\`atica
Aplicada II, Universitat Polit\`ecnica de Catalunya, C/Jordi
Girona, 1-3 Edifici Omega, E-08034 Barcelona, Spain, e-mail:
ana.rio@upc.edu

\vspace{0.3cm} \noindent Montserrat Vela, Departament de
Matem\`atica Aplicada II, Universitat Polit\`ecnica de Catalunya,
C/Jordi Girona, 1-3 Edifici Omega, E-08034 Barcelona, Spain,
e-mail: \linebreak montse.vela@upc.edu

\end{document}